\documentclass{amsart}
\usepackage{amssymb,amstext,amsmath,amscd,amsthm,amsfonts,enumerate,graphicx,latexsym,
color,stmaryrd, amsxtra}
\usepackage{mathptmx}
\usepackage[notcite, notref]{}
\usepackage{nicefrac}

\usepackage[usenames,dvipsnames]{pstricks}
\usepackage[mathscr]{eucal}

\usepackage[all]{xy}

\newtheorem{thm}{Theorem}[section]
\newtheorem*{thm*}{Theorem}

\newtheorem{cor}[thm]{Corollary}

\newtheorem{prop}[thm]{Proposition}
\theoremstyle{definition}

\newtheorem{dfn}[thm]{Definition}
\newtheorem*{dfn*}{Definition}

\newtheorem{chunk}[thm]{}
\newtheorem*{conj*}{Conjecture}
\newtheorem{ex}[thm]{Example}

\theoremstyle{remark}

\newtheorem*{claim*}{Claim}

\numberwithin{equation}{thm}

\def\Tr{\mathsf{Tr}}

\def\Hom{\mathsf{Hom}}
\def\im{\operatorname{\mathsf{im}}}
\def\ker{\operatorname{\mathsf{ker}}}

\newcommand{\ZZ}{\mathbb{Z}}

\def\depth{\mathsf{depth}}

\def\Gdim{\textnormal{G-dim}}
\def\Hdim{\textnormal{H-dim}}
\def\H{\textnormal{H}}

\def\pd{\mathsf{pd}}
\def\id{\mathsf{id}}

\def\Ext{\mathsf{Ext}}
\def\Tor{\mathsf{Tor}}

\def\fm{\mathfrak{m}}
\def\fp{\mathfrak{p}}

\def\sH{\mathsf{H}}

\title{Powers of the maximal ideal and vanishing of (co)homology}

\author[Celikbas]{Olgur Celikbas}
\address[Olgur Celikbas]{
Department of Mathematics, West Virginia University, Morgantown, WV 26506-6310, USA}
\email{olgur.celikbas@math.wvu.edu}

\author[Takahashi]{Ryo Takahashi} 
\address[Ryo Takahashi]{Graduate School of Mathematics, Nagoya University, Furocho, Chikusaku, Nagoya, Aichi 464-8602, Japan/Department of Mathematics, University of Kansas, Lawrence, KS 66045-7523, USA}
\email{takahashi@math.nagoya-u.ac.jp}
\urladdr{https://www.math.nagoya-u.ac.jp/~takahashi/}
\thanks{Takahashi was partly supported by JSPS Grants-in-Aid for Scientific Research 16K05098 and 16KK0099}

\begin{document}

\begin{abstract} We prove that each positive power of the maximal ideal of a commutative Noetherian local ring is Tor-rigid, and strongly-rigid. This gives new characterizations of regularity and, in particular, shows that such ideals satisfy the torsion condition of a long-standing conjecture of Huneke and Wiegand. 
\end{abstract}

\maketitle

\section{Introduction}

Throughout $R$ denotes a commutative Noetherian local ring with unique maximal ideal $\fm$ and residue field $k$, and all $R$-modules are assumed to be finitely generated.

In this paper we are motivated by the following result of Levin and Vasconcelos:

\begin{thm}[{Levin and Vasconcelos \cite{LV}}] \label{LV} Let $M$ be an $R$-module. Assume $\fm^tM\neq 0$ for some $t\geq 1$. Then $R$ is regular if and only if $\pd_R(\fm^tM)<\infty$ if and only if $\id_R(\fm^tM)<\infty$. 
\end{thm}

Theorem \ref{LV} was examined previously in the literature. For example, Asadollahi and Puthenpurakal obtained beautiful characterizations of local rings in terms of various homological dimensions: if  $M$ is an $R$-module of positive depth with $\Hdim_R(\fm^n M)<\infty$ for some $n\gg 0$, then $R$ satisfies the property $\H$, where $\Hdim$ denotes a homological dimension such as projective dimension; see \cite[Theorem 1]{TJ} for details. The main purpose of this short note is to prove an analogous result for nonzero modules of the form $\fm^tM$. However, our main result, stated as Theorem \ref{mainthmintro}, concerns the vanishing of $\Ext$ and $\Tor$ rather than homological dimensions. 

\begin{thm} \label{mainthmintro} Let $M$ be an $R$-module with $\depth_R(M)\geq 1$, and let $t\geq 0$ be an integer. If $\Tor_n^R(\fm^t M, N) = 0$ (respectively, $\Ext_R^n(N, \fm^t M) = 0$) for some $R$-module $N$ and some $n\geq 1$, then $\Tor_n^R(M, N) = 0$ (respectively, $\Ext_R^n(N, M) = 0$). 
\end{thm}

Theorem \ref{mainthmintro} does not hold for modules of zero depth in general; see Examples \ref{ornek1} and \ref{ornek2}. In section 2 we give a proof of Theorem \ref{mainthmintro} and discuss its consequences. We should mention that one such consequence of Theorem \ref{mainthmintro} is Theorem \ref{LV} in the positive depth case; see the paragraph after Example \ref{ornek2}. Moreover,  Theorem \ref{mainthmintro} implies that each positive power of the maximal ideal is Tor-rigid and strongly-rigid; see Definition \ref{TTdfn}. More precisely, we have:

\begin{cor} \label{sonuc1} Assume $\depth(R)\geq 1$, and let $t\geq 1$. If $\Tor_n^R(\fm^t, N)=0$ for some $n\geq 0$ and some $R$-module $N$, then $\pd_R(N)\leq n$, and hence $\Tor_i^R(\fm^t, N)=0$ for all $i\geq n$.
\end{cor}

Although the behavior of powers of the maximal ideal obtained in Corollary \ref{sonuc1} may seem expectable, to the best of our knowledge, the conclusion of the corollary is new. Note that Corollary \ref{sonuc1} follows immediately by letting $M=\fm$ in Theorem \ref{mainthmintro}. 

Corollary \ref{sonuc1} yields a new characterization of regularity for local rings of positive depth: we record the result as Corollary \ref{sonuc2} and prove it in the paragraph preceding Corollary \ref{cor1}.

\begin{cor} \label{sonuc2} Assume $\depth(R)\geq 1$. If $M$ is an $R$-module such that $\fm^sM\neq 0$ for some $s\geq 1$, then $R$ is regular if and only if $\Ext^n_R(\fm^sM,\fm^t )=0$ for some $n, t\geq 1$. In particular, $R$ is regular if and only if $\Ext^n_R(\fm^s,\fm^t)=0$ for some $n, s, t\geq 1$. 
\end{cor}

As another consequence of Corollary \ref{sonuc1}, we conclude by \cite[2.15]{CGTT} that each positive power of the maximal ideal satisfies the torsion condition proposed in a long-standing conjecture of Huneke and Wiegand; see \cite[pages 473-474]{HW1} for details.

\begin{cor} \label{sonuc3} Assume $R$ is one-dimensional, non-regular, and reduced. Then $\fm^t \otimes_R (\fm^t)^{\ast}$ has torsion for each $t\geq 1$, where $(\fm^t)^{\ast}=\Hom_R(\fm^t,R)$.
\end{cor}

Levin and Vascencelos \cite[Lemma, page 316]{LV} proved, if $M$ and $N$ are $R$-modules such that $\fm M\neq 0$ and $\Tor_n^R(\fm M, N)=\Tor_{n+1}^R(\fm M, N)=0$ for some $n\geq 0$, then $\pd_R(N)\leq n$ and $\Tor_i^R(\fm M, N)=0$ for all $i\geq n$; see also \cite[2.9]{ARideal}. While proving Theorem \ref{mainthmintro}, we have discovered that we can extend the result of Levin and Vascencelos  by considering nonzero modules of the form $\fm M\otimes_R N$: at the end of Section 2, we will observe that each such module is isomorphic to $\fm C$ for some $R$-module $C$, and we will prove the following:

\begin{prop} \label{prop1} Assume $R$ is not Artinian and let $n\geq 1$ be an integer. Then the following conditions are equivalent:
\begin{enumerate}[\rm(i)]
\item $R$ is Gorenstein.
\item $\Ext^i_R(\fm^{\otimes n}, R)=0$ for all $i\gg 0$.
\item $\Ext^i_R(\fm^{n}, R)=0$ for all $i\gg 0$.
\end{enumerate}
\end{prop}

\section{Proof of the main result and consequences}

We start by recalling some definitions.

\begin{dfn} \label{TTdfn} Let $M$ be an $R$-module. Recall that:
\begin{enumerate}[\rm(i)]
\item (\cite{Au}) $M$ is \emph{Tor-rigid} provided that the following holds: whenever $N$ is an $R$-module with $\Tor_j^R(M, N)=0$ for some $j\geq 1$, one has that $\Tor_{v}^R(M, N)=0$ for all $v\geq j$.
\item (\cite{DLM}) $M$ is \emph{strongly-rigid} provided that the following holds: whenever $N$ is an $R$-module with $\Tor_{n}^{R}(M,N)=0$ for some $n\geq 1$, one has that $\pd_R(N)<\infty$.
\end{enumerate}
\end{dfn}

Let us point out that, it is not known whether strongly-rigid modules are Tor-rigid; see \cite[Question 2.5]{RM}. In general it is quite subtle to determine whether a given module is strongly-rigid or Tor-rigid, but various characterizations of local rings have already been obtained in terms of such classes of modules. For example, existence of a nonzero Tor-rigid module of finite injective dimension forces the ring to be Gorenstein; see \cite[4.13(i)]{RM}. 

The following, straightforward albeit quite useful, observation is implicit in \cite{LV}.

\begin{chunk}\label{mx} Let $C=(C_i, \partial_i)_{i\in \ZZ}$ be a minimal complex with $C_i$ are $R$-modules, i.e., $\im (\partial_{i+1}) \subseteq \fm \cdot C_i$ for each $i$.
Assume $\sH_n(\fm C) = 0$ for some $n \in \ZZ$. As $\sH_n(\fm C) = \ker(\partial_n) \cap \fm C_n / \fm \cdot \im(\partial_{n+1})$, we have
$ \im(\partial_{n+1}) \subseteq \ker(\partial_n) \cap \fm \cdot C_n = \fm \cdot \im(\partial_{n+1})$.
By Nakayama's lemma, we conclude that $\im(\partial_{n+1}) = 0$, i.e., $\partial_{n+1} = 0$.
\end{chunk}

We can now use \ref{mx} and prove our main result:

\begin{proof}[Proof of Theorem \ref{mainthmintro}] \label{prf} We will only prove the statement about the vanishing of $\Tor$; the one about $\Ext$ follows similarly.

Note that $\depth_R (\fm^j M)\geq 1$ for any $j \ge 0$. Hence it suffices to consider the case where $t=1$ and $n\geq 1$. Assume $\Tor_n^R(\fm M, N) = 0$, and consider the exact sequence $$0 \to \fm M \to M \to M/\fm M \to 0.$$ This yields the exact sequence $0 = \Tor_n^R(\fm M, N) \to \Tor_n^R(M, N) \to \Tor_n^R(M/\fm M, N)$, which shows that $\Tor_n^R(M, N)$ has finite length.

Let $C=M \otimes_R F$, where $F=(F_i,\partial_i^F)_{i\ge0}$ is a minimal free resolution of $N$. It follows that $\im(\partial_{i+1}^{C}) \subseteq \fm \cdot C_i$ for each $i$. Since $0=\Tor_n^R(\fm M, N) = \sH_n(\fm M \otimes_R F)$, we see from \ref{mx} that $\partial_{n+1}^{C} = 1_M \otimes_R \partial_{n+1}^F = 0$. Therefore, we have
$$
\Tor_n^R(M, N) = \ker\left(1_M \otimes_R \partial_n^F \right) / \left( \im (1_M \otimes_R \partial_{n+1}^F)\right) = \ker\left(1_M \otimes_R \partial_n^F \right).
$$
Now suppose $\Tor_n^R(M, N)\neq 0$. Then, since it embeds into a finite direct sum of copies of $M$, we conclude that $1\leq \depth_R\left(\Tor_n^R(M, N)\right)<\infty$. However, $\Tor_n^R(M, N)$ has finite length so that $\depth_R\left(\Tor_n^R(M, N)\right)=0$. This shows that $\Tor_n^R(M, N)$ must vanish, as claimed.
\end{proof}

It is also worth noting that Theorem \ref{mainthmintro} may fail if the module in question has zero depth: we give two such examples over rings of depth one and two, respectively.

\begin{ex} \label{ornek1} Let $k$ be a field, $R=k[\![x,y ]\!] /(xy)$, $M=k \oplus R$, and let $N=R/(x+y)$. Then $\depth_R(M)=0$, $\fm M =\fm$ and $\pd_R(N)=1$. However, $\Tor_1^R(\fm M, N)=0 \neq \Tor_1^R(M, N)$.
\end{ex}

\begin{ex} \label{ornek2} Let $k$ be a field, $R=k[\![x,y ]\!]$, $\fm=(x,y)$, $M=\fm/(x^2, xy)$ and $N=R/(y)$. Then $\fm M= \fm^2/(x^2, xy) \cong R/\left( (x^2, xy):_Ry^2 \right)=R/(x)$. Hence $\Tor_1^R(\fm M, N)=0$.
But $\Tor_1^R(M,N)\neq 0$ since multiplication by $y$ on $M$ is not injective. Note that $\depth_R(M)=0$.
\end{ex}

Next we discuss several corollaries of Theorem \ref{mainthmintro}. First we deduce from Theorem \ref{mainthmintro} the positive depth case of Theorem \ref{LV}, and prove Corollary \ref{sonuc2}.

\begin{proof}[Proof of the positive depth case of Theorem 1.1 by using Theorem 1.2] Assume $M$ has positive depth and $\pd_R(\fm^tM)$ is finite, say $s$. Then we have $\Tor_{s+1}^R(\fm^tM,k)=0$ so that $\Tor_{s+1}^R(\fm^{t-1}M,k)=0$ by Theorem 1.2. Hence, $\pd_R(\fm^{t-1}M)$ is also finite. There is an exact sequence $$0 \to \fm^tM \to \fm^{t-1}M \to k^{\oplus u} \to 0,$$ with $u\ge1$ as $\fm^{t-1}M\ne0$. It follows that $\pd_R(k)<\infty$, and $R$ is regular. The assertion on injective dimension is shown similarly. 
\end{proof}

\begin{proof}[Proof of Corollary \ref{sonuc2}] Assume $\fm^sM\neq 0$ and $\Ext^n_R(\fm^sM,\fm^t )=0$ for some $n, s, t\geq 1$. Note, since $\depth(R)\geq 1$, Corollary \ref{sonuc1} implies that $\fm^t$ is strongly-rigid and Tor-rigid. As $\depth_R(\fm^t)=1$, we conclude from \cite[1.1]{RM} that $\pd_R(\fm^sM)<\infty$. Now Theorem \ref{LV} shows that $R$ is regular.
\end{proof}

\begin{cor} \label{cor1} Let $M$ be an $R$-module such that $\depth_R(M)\geq 1$.
\begin{enumerate}[\rm(i)]
\item If $M$ strongly-rigid, then $\fm^t M$ is strongly-rigid for each $t\geq 1$. 
\item If $M$ is Tor-rigid, then $\fm M$ is strongly-rigid and Tor-rigid.
\end{enumerate}
\end{cor}

\begin{proof} Part (i) is an immediate corollary of Theorem \ref{mainthmintro}. So we will prove part (ii).

Let $N$ be an $R$-module with $\Tor_n^R(\fm M, N)=0$ for some $n\geq 1$. Then it follows from Theorem \ref{mainthmintro} that $\Tor_n^R(M, N)=0$. Since $M$ is Tor-rigid,  we have that $\Tor_{n+1}^R(M, N)=0$. Hence, tensoring the exact sequence $0 \to \fm M \to M \to M/\fm M \to 0$ with $N$, we obtain the exact sequence $0=\Tor_{n+1}^R(M, N) \to \Tor_{n+1}^R(M/\fm M, N) \to \Tor_n^R(\fm M, N)=0$.
This shows $\Tor_{n+1}^R(k, N)=0$ so that $\pd_R(N)\leq n$, and $\Tor_i^R(\fm M, N)=0$ for all $i\geq n$.
\end{proof}

Theorem \ref{mainthmintro} allows us to find out new classes of Tor-rigid modules over hypersurfaces:

\begin{cor} Let $R=S/(f)$ be a hypersurface ring, where $(S, \mathfrak{n})$ is an unramified regular local ring and $0\neq f \in \mathfrak{n}$. If $M$ is a finite length $R$-module, then $\fm \Omega^i(M)$ is Tor-rigid for each $i\geq 1$.
\end{cor}

\begin{proof} Note that each finite length module is Tor-rigid \cite[2.4]{HW1}. Hence we may assume $R$ has positive depth. Then, given $i\geq 1$, since $\Omega^i(M)$ is a Tor-rigid module that has positive depth, we conclude by Corollary \ref{cor1}(ii) that $\fm \Omega^i(M)$ is Tor-rigid, as claimed.
\end{proof}

If $R$ has positive depth and $N$ is an $R$-module, it is known, and easy to see, that $N$ is free if and only if  $\fm  \otimes_R N$ is torsion-free; see, for example, \cite[page 842]{Const}. Thanks to Theorem \ref{mainthmintro}, we can extend this result under mild conditions: 

\begin{cor} \label{CMmaincor2} Assume $\depth(R)\geq 1$, and let $N$ be an $R$-module. Assume $N_{\fp}$ is torsionless for each associated prime ideal $\fp$ of $R$ (e.g., $R$ is reduced). Then $N$ is free if and only if $\fm ^t \otimes_R N$ is torsion-free for some $t\geq 1$. 
\end{cor}

\begin{proof} Let $X$ be the torsion-free part of $N$. Then $X_{\fp} \cong N_{\fp}$ for each associated prime $\fp$ of $R$. So $\Ext^{1}_{R}(\Tr X, R)=0$, and hence there is an exact sequence $0 \to X \to F \to C \to 0$, where $F$ is a free module; see, for example, \cite[Prop. 5]{Masek}.

Tensoring $X$ with the short exact sequence $0 \to \fm^t \to R \to R/\fm^t \to 0$, we conclude that there is an injection $\Tor_1^R(X, R/\fm^t) \hookrightarrow \fm^t \otimes_R X \cong \fm^t \otimes_R N$; see \cite[1.1]{HW1}. This implies that $\Tor_1^R(X, R/\fm^t) =0$. Therefore we have $0=\Tor_2^R(C, R/\fm^t) \cong \Tor_1^R(C, \fm^t) $, and hence $\pd_R(C)\leq 1$; see Corollary \ref{sonuc1}. Thus, $X$ is free and this implies $N$ is free; see \cite[1.1]{HW1}. 
\end{proof}


We finish this section by showing that modules of the form $\fm M \otimes_RN$ is strongly-rigid and Tor-rigid. This will allow us to establish Proposition \ref{prop1} advertised in the introduction.

\begin{chunk} \label{obs2} Let $M$ and $N$ be $R$-modules such that $\fm M \neq 0 \neq N$. Then consider the minimal free presentations of $M$ and $N$, respectively: $R^{\oplus a} \twoheadrightarrow M$ and $R^{\oplus b} \twoheadrightarrow N$.

It follows that we have the surjection: $\fm^{\oplus a} = \fm R^{\oplus a} \twoheadrightarrow \fm M$. Tensoring this surjection with $N$, we obtain another surjection: $\fm^{\oplus a} \otimes_R N\twoheadrightarrow \fm M\otimes_R N$. Consequently, we have the following isomorphisms and surjective maps:
$$\fm R^{\oplus ab} = \fm^{\oplus ab} \cong \fm^{\oplus a} \otimes_R R^{\oplus b}\twoheadrightarrow \fm^{\oplus a} \otimes_R N \twoheadrightarrow \fm M\otimes_R N.$$
Therefore, there is an $R$-submodule $C$ of $\fm R^{\oplus ab}$ such that 
$$\fm M \otimes_R N \cong \frac{\fm R^{\oplus ab}}{C} = \fm \left(\frac{R^{\oplus ab}}{C}\right).$$
\end{chunk} 

The rigidity property (mentioned preceding Proposition \ref{prop1}) of nonzero modules of the form $\fm M$, in view of \ref{obs2}, yields:

\begin{chunk} \label{obs3} If $\Tor_n^R(\fm M\otimes_R N, X)=\Tor_{n+1}^R(\fm M\otimes_RN, X)=0$ for some $R$-modules $M$, $N$, $X$ and $n\geq 0$ such that $\fm M \neq 0 \neq N$, then $\pd_R(X)\leq n$, and $\Tor_i^R(\fm M\otimes_RN, X)=0$ for all $i\geq n$.
\end{chunk}

The observations in \ref{obs2} and \ref{obs3}, in particular, show that tensor powers of the maximal ideal have rigidity:

\begin{chunk}  \label{obs4} Assume $R$ is not Artinian, $t\geq 1$ and $\fm^{\otimes0}=R$. Then, letting $M=R$ and $N=\fm^{\otimes (t-1)}$ in \ref{obs3}, we conclude that, if $\Tor_n^R(\fm ^{\otimes t}, X)=\Tor_{n+1}^R(\fm ^{\otimes t}, X)=0$ for some $R$-module $X$ and some $n\geq 0$, then $\pd_R(X)\leq n$, and $\Tor_i^R(\fm ^{\otimes t}, X)=0$ for all $i\geq n$.
\end{chunk}

We can now note that Proposition \ref{prop1} is a consequence of \ref{obs4} and \cite[4.4]{CelWag}. We finish this section by recording a special case of Proposition \ref{prop1}:

\begin{prop} \label{prop2} $R$ is Gorenstein if and only if $\Gdim_R(\fm \otimes_R \fm)<\infty$.
\end{prop}

\end{document}